\documentclass[12pt]{article}
\usepackage[a4paper,margin=1.0in]{geometry}
\usepackage[colorlinks,citecolor=magenta,linkcolor=black]{hyperref}
\pdfpagewidth=\paperwidth \pdfpageheight=\paperheight
\usepackage{amsfonts,amssymb,amsthm,amsmath,eucal,tabu,url}
\usepackage{pgf}
 \usepackage{array}
 \usepackage{tikz-cd}
 \usepackage{pstricks}
 \usepackage{pstricks-add}
 \usepackage{pgf,tikz}
 \usetikzlibrary{automata}
 \usetikzlibrary{arrows}
 \usepackage{indentfirst}
\usepackage{tabularx} 
\usepackage{longtable}

\theoremstyle{plain}
\newtheorem{thm}{Theorem}[section]
\newtheorem{theorem}[thm]{Theorem}

\newtheorem{proposition}[thm]{Proposition}
\newtheorem{corollary}[thm]{Corollary}

\theoremstyle{definition}
\newtheorem{definition}[thm]{Definition}
\newtheorem{remark}[thm]{Remark}
\newtheorem{example}[thm]{Example}

\newtheorem{thevarthm}[thm]{\varthmname}

\newenvironment{varthm*}[1]{\trivlist\item[]{\bf #1.}\it}{\endtrivlist}
\def\keywordname{{\bfseries Keywords}}%
\def\keywords#1{\par\addvspace\medskipamount{\rightskip=0pt plus1cm
\def\and{\ifhmode\unskip\nobreak\fi\ $\cdot$
}\noindent\keywordname\enspace\ignorespaces#1\par}}

\def\subclassname{{\bfseries Mathematics Subject Classification
(2020)}\enspace}
\def\subclass#1{\par\addvspace\medskipamount{\rightskip=0pt plus1cm
\def\and{\ifhmode\unskip\nobreak\fi\ $\cdot$
}\noindent\subclassname\ignorespaces#1\par}}

\begin{document}
\title{On combinatorics of plus--one generated line arrangements}
\author{Artur Bromboszcz}
\date{\today}
\maketitle

\thispagestyle{empty}
\begin{abstract}
This paper focuses on the combinatorial properties of plus--one generated line arrangements. We provide combinatorial constraints on such arrangements and construct a polynomial with a shape similar to the well--known Poincar\'e polynomial that decodes the plus--one generatedness property. We demonstrate how to create new plus--one generated arrangements using classical Klein and Wiman reflection arrangements. Furthermore, among all known sporadic simplicial arrangements with up to $27$ lines, we identify nine minimal plus--one generated arrangements.
\subclass{14N20, 32S22, 14C20}
\end{abstract}
\section{Introduction}
In this paper, we study line arrangements that are plus--one generated, focusing primarily on their combinatorial aspects. Abe recently introduced the concept of plus--one generated hyperplane arrangements in \cite{Abe18}, and these arrangements are closely related to free arrangements. Recall that an arrangement of hyperplanes is called {\it free} if its associated module of derivations is a free module over the polynomial ring. For plus--one generated arrangements, the associated module of derivations is no longer free. However, it admits a simple minimal free resolution. This is a natural step away from the freeness property, which is why the world of plus--one generated arrangements attracts a lot of attention from researchers nowadays. We should also mention that Dimca and Sticlaru generalized the notion of plus--one generated hyperplane arrangements to the setting of reduced plane curves, see \cite{DS0}. In this broader context, we have an analogous result to Abe's theorem on hyperplane arrangements, namely if $C \subset \mathbb{P}^{2}_{\mathbb{C}}$ is a reduced plane curve containing a line, say $\ell$, as an irreducible component, then assuming $C$ is free, the deletion $C' = C \setminus \ell$ is either free or plus--one generated, see \cite{POG, MP}. The main aim of this paper is to present combinatorial properties of plus--one generated line arrangements and to apply new techniques to construct examples of such arrangements. Let us present briefly an outline of our work. First of all, we provide some combinatorial constraints on plus--one generated line arrangements. In particular, we show that plus--one generated arrangements of $d>12$ lines have to have at least one intersection point of multiplicity at least $4$, see Corollary \ref{max}. Our main result of the paper is presented in Section $4$, where we construct a Poincar\'e-type polynomial based on ideas of Pokora presented in \cite{Pok}. More precisely, given an arrangement $\mathcal{L} : f=0$ of $d$ lines in $\mathbb{P}^{2}_{\mathbb{C}}$, we define
$$\mathcal{P}(\mathcal{L},d_{3};t) = 1 + dt + \bigg( \sum_{r\geq 2}(r-1)t_{r} - d_{3}\bigg)t^{2},$$
where $t_{r}$ denotes the number of $r$-fold intersection points and $d_{3}$ denotes the degree of the third syzygy of the Jacobian ideal $J_f$ with respect to the degree order. We show that it behaves in a similar way to the classical Poincar\'e polynomial with this significant change that it decodes the property of being plus--one generated. In this setting, Theorem \ref{npog} gives us a combinatorial criterion that allows us to decide if a given line arrangement cannot be plus--one generated. We present some examples that demonstrate the utility of this technique in Remark \ref{pp1} and \ref{pp2}. In the last part of the paper we show how to construct new examples of plus--one generated arrangements using certain classical reflection line arrangements and the deletion technique. For example, we show that the deletion of any line from the Klein arrangement of $21$ lines leads us to a minimal plus--one generated line arrangement, see Proposition \ref{klein}. Finally, we focus on sporadic simplicial line arrangements in $\mathbb{P}^{2}_{\mathbb{R}}$. Our main contribution is demonstrating that, of all the sporadic simplicial line arrangements with up to $27$ lines, exactly nine are minimal plus--one generated, see Theorem \ref{spog}. In the paper we performed symbolic computations using \verb}SINGULAR} \cite{Singular}.
\section{Preliminaries}
We will recall some basic statements from the theory of plane curves essential for the present note. Let $C$ be a reduced curve in $\mathbb{P}^{2}_{\mathbb{C}}$ of degree $d$ given by $f \in S :=\mathbb{C}[x,y,z]$. We denote by $J_{f}$ the Jacobian ideal, i.e., the ideal generated by the partial derivatives $\partial_{x}f, \, \partial_{y}f, \, \partial_{z}f$. We define $r:={\rm mdr}(f)$ to be the minimal degree of a syzygy relation among the partial derivatives, i.e., 
$${\rm mdr}(f):=\min\{r : \, {\rm AR}(f)_{r} \neq 0 \},$$
where ${\rm AR}(f) = \{(a,b,c)\in S^{\oplus 3} : af_x+bf_y+cf_z=0\}$.
Finally, we define the Milnor algebra as $M(f) = S / J_{f}$.
\begin{definition}
We say that a reduced plane curve $C$ is an $m$-syzygy curve when the associated Milnor algebra $M(f)$ has the following minimal graded free resolution
$$0 \rightarrow \bigoplus_{i=1}^{m-2}S(-e_{i}) \rightarrow \bigoplus_{i=1}^{m}S(1-d - d_{i}) \rightarrow S^{3}(1-d)\rightarrow S \rightarrow M(f) \rightarrow 0$$
with $e_{1} \leq e_{2} \leq \ldots \leq e_{m-2}$ and $1\leq d_{1} \leq \ldots \leq d_{m}$. The $m$-tuple $(d_{1}, \ldots , d_{m})$ is called the exponents of $C$.
\end{definition}

\begin{definition}
A reduced curve $C$ in $\mathbb{P}^{2}_{\mathbb{C}}$ is called \textbf{plus--one generated} with the exponents $(d_1,d_2,d_3)$ if $C$ is $3$-syzygy such that $d_{1}+d_{2}=d$. 
\end{definition}
In order to study plus--one generated line arrangements we will use the following characterization that comes from \cite{DS0}. Here by $\tau(C)$ we denote the total Tjurina number of a given reduced curve $C \subset \mathbb{P}^{2}_{\mathbb{C}}$, which is defined as the degree of the Jacobian ideal $J_{f}$, i.e., $\tau(C) = {\rm deg}(J_{f})$.
\begin{proposition}[Dimca-Sticlaru]
\label{dimspl}
Let $C: f=0$ be a reduced $3$-syzygy curve of degree $d\geq 3$ with the exponents $(d_{1},d_{2},d_{3})$. Then $C$ is plus--one generated if and only if
$$\tau(C) = (d-1)^{2} - d_{1}(d-d_{1}-1) - (d_{3}-d_{2}+1).$$
\end{proposition}
\noindent
Now we will define minimal plus--one generated curves that have been introduced in \cite{DS24}.
\begin{definition}
A plus--one generated plane curve $C$ satisfying
\begin{itemize}
    \item $d_3=d_2$ is called \textbf{nearly free},
    \item $d_3=d_2+1$ is called \textbf{minimal plus one--generated}.
\end{itemize}
\end{definition}

To check if a reduced curve is a minimal plus--one generated, we can use the following characterization by Dimca and Sticlaru.
\begin{theorem}[{\cite[Theorem 1.5]{DS24}}]
\label{MPOG}
Let $C \, : f=0$ be a reduced plane curve of degree $d$ in $\mathbb{P}^{2}_{\mathbb{C}}$ with $r = {\rm mdr}(f) \leq d/2$. Then $C$ is a minimal plus--one generated curve if and only if
$$r^{2} - r(d-1) + (d-1)^2 = \tau(C)+2.$$
\end{theorem}
\noindent
As we will see in a moment, this criterion is very handy and easy to use in practice.
\section{On combinatorics of plus--one generated line arrangements}
We present our results on the combinatorial properties of plus--one generated line arrangements. We decided to focus on this aspect because the results reveal new insights into the connections between combinatorics and the homological aspects of plus--one generatedness.

Our first result is about the maximal multiplicity of plus--one generated line arrangements.
\begin{proposition}
\label{mm}
    Let $\mathcal{L} \subset \mathbb{P}^{2}_{\mathbb{C}}$ be a plus--one generated arrangement of $d$ lines. Then
    $$m(\mathcal{L})\geq \left\lceil \frac{4d}{d+4} \right\rceil,$$
    where $m(\mathcal{L})$ is the maximal multiplicity among intersection points in $\mathcal{L}$
\end{proposition}

\begin{proof}
If $\mathcal{L}$ is plus--one generated with $d$ lines with exponent $
(d_{1}, d_{2}, d_{3})$, then $d\geq 2d_1$, and this follows from the fact that $d_{1}+d_{2}=d$ and $d_{1} \leq d_{2}$. Moreover, using \cite[Theorem 2.1]{Dimca2}, we get $$\frac{d}{2} \geq d_1 \geq \frac{2}{m(\mathcal{L})}\cdot d-2.$$
After simple manipulations, we arrive at
$m(\mathcal{L})\geq\frac{4d}{d+4}$.
Since $m(\mathcal{L})\in\mathbb{N}$, we obtain 
$$m(\mathcal{L})\geq\left\lceil\frac{4d}{d+4}\right\rceil.$$
\end{proof}
\begin{corollary}
\label{max}
In the setting of Proposition \ref{mm}, if $\mathcal{L} \subset \mathbb{P}^{2}_{\mathbb{C}}$ is a plus--one generated arrangement of $d>12$ lines, then $m(\mathcal{L}) \geq 4$.
\end{corollary}

The following combinatorial identity will play a crucial role in the forthcoming section.
\begin{thm}
\label{comeq}
     Let $\mathcal{L} \subset \mathbb{P}^{2}_{\mathbb{C}}$ be a plus--one generated arrangement of $d$ lines with the exponents $(d_1, d_2, d_3)$. Then
     $$\sum_{r\geq2}(r-1)\cdot t_r=d_1 d_2+d_3,$$
where $t_r = t_{r}(\mathcal{L})$ is the number of $r$-fold intersection points of $\mathcal{L}$.
\end{thm}

\begin{proof}
    Recall that if $\mathcal{L}$ is line arrangement, then all singularities are quasi--homogeneous and $$\tau(\mathcal{L})=\sum_{r\geq 2}{(r-1)^2\cdot t_r},$$
    where $\tau(\mathcal{L})$ is the total Tjurina number.
    Combining above formula with the naive count, that 
    $$d^2-d=\sum_{r\geq 2}{(r^2-r)\cdot t_r},$$
    we obtain
    $$\tau(\mathcal{L})=\sum_{r\geq 2}(r^2-2r+1)\cdot t_r=\sum_{r\geq 2}(r^2-r)\cdot t_r-\sum_{r\geq 2}(r-1)\cdot t_r=d^2-d-\sum_{r\geq 2}(r-1)\cdot t_r.$$
    Since $\mathcal{L}$ is plus--one generated, we can use Proposition 2.3, and this gives us
    $$\sum_{r\geq 2}(r-1)\cdot t_r=d+d_1d-d_1^2-d_1+d_3-d_2.$$
    The assumption of plus--one generatedness implies $d=d_1+d_2$, and this finally allows us to conclude that 
    $$\sum_{r\geq 2}(r-1)\cdot t_r=d_1+d_2+d_1^2+d_1d_2-d_1^2-d_1+d_3-d_2=d_1d_2+d_3,$$
    which completes the proof.

\end{proof}
\begin{remark}
Using \cite[Corollary 3.5]{Sch}, let us observe that for plus--one generated line arrangements, one has $d_{3}\leq d-2 = d_{1}+d_{2}-2$, so we get
$$\sum_{r\geq 2}(r-1)\cdot t_r=d_1d_2+d_3 \leq (d_{1}+1)(d_{2}+1)-3.$$
\end{remark}
The next result provides a non--trivial lower bound on the number of intersections of plus--one-generated line arrangements using only the homological language of exponents.
\begin{proposition}
Let $\mathcal{L}$ be a plus--one generated arrangement of $d\geq 6$ lines with exponents $(d_{1},d_{2},d_{3})$. Then
$$\sum_{r\geq 2} t_{r} \geq \bigg\lceil\frac{1}{3}(d_{1}d_{2} + d_{1}+d_{2}+d_{3})\bigg\rceil.$$
\end{proposition}
\begin{proof}
Since $\mathcal{L}$ is an arrangement of $d\geq 6$ lines with $t_d=t_{d-1}=0$, then by Hirzebruch's inequality \cite{Hirzebruch} one has $$t_{2} + t_{3} \geq d + \sum_{r\geq 4}(r-4)t_{r}.$$
Observe that we can freely use the above inequality since arrangements with $t_{d}\neq 0$ or $t_{d-1}\neq 0$ are automatically free as they are supersolvable line arrangements, see for instance \cite{DS}. Using Theorem 3.3, and the assumption that $d=d_1+d_2$ for plus--one generated arrangements, we have
$$3  \sum_{r\geq 2}t_{r} \geq d+t_{2} + \sum_{r\geq 2}(r-1)t_{r}=d_1d_2+d_1+d_2+d_3+t_2,$$
which gives us
$$\sum_{r\geq 2}{t_r} \geq \frac{1}{3}\left( d_{1}d_{2}+d_{1}+d_{2}+d_{3}\right).$$
As $\sum_{r\geq 2}{t_r}$ is a positive integer, we can take the ceiling, which completes the proof.


\end{proof}
Now we focus on the defect of plus--one generated line arrangements. We should point out here that the defect of a reduced plane curve is an invariant that measures the discrepancy from the property of being free. Roughly speaking, free plane curves have defect equal to $0$, and in the case of plus--one generated plane curves $C$ the defect $\nu(C)$ is equal to $d_{3}-d_{2}+1$. Our goal now is to present a straightforward proof of this fact within the context of line arrangements, building upon the combinatorial results obtained earlier. We will need the following result by Dimca.
\begin{thm}[{\cite[Theorem 1.2]{Dimca1}}]
\label{deff}
Let $C \, : f=0$ be a reduced plane curve of degree $d$ and $d_{1}= {\rm mdr}(C)$. Then the following hold.
\begin{enumerate}
    \item[a)] If $d_{1} < d/2$, then $\nu(C) = (d-1)^{2}-d_{1}(d-1-d_{1})-\tau(C)$.
    \item[b)] If $d_{1} \geq d/2$, then
    $$\nu(C) = \bigg\lceil\frac{3}{4}(d-1)^{2}\bigg\rceil - \tau(C).$$
\end{enumerate}
\end{thm}
We are now in a position to prove the announced result.
\begin{proposition}
Let $\mathcal{L} \subset \mathbb{P}^{2}_{\mathbb{C}}$ be a plus--one generated arrangement of $d\geq 4$ lines with exponents $(d_{1},d_{2},d_{3})$, then $\nu(\mathcal{L}) = d_{3}-d_{2}+1$.
\end{proposition}
\begin{proof}
Depending on $d_1$, we have to consider two cases. First, let us consider the case when $d_1 < d/2$. In this case, we can apply Theorem \ref{deff} a) directly, which tells us that the defect of $\mathcal{L}$ is equal to
$$\nu(\mathcal{L}) = (d-1)^{2}-d_{1}(d-1-d_{1})-\tau(\mathcal{L}).$$
Observe that in the proof of Theorem \ref{comeq} we obtained 
$$\tau(\mathcal{L}) = d^{2}-d - \sum_{r\geq 2}(r-1)t_{r},$$
and since $\sum_{r\geq 2}(r-1)t_{r} = d_{1}d_{2}+d_{3}$ we arrive at
\begin{equation}
\tau(\mathcal{L}) = d^{2}-d - d_{1}d_{2} - d_{3} = d^{2}-d - d_{1}(d-d_{1})-d_{3}.
\end{equation}
When we plug this formula into the formula for the defect, we get
$$\nu(\mathcal{L}) =  (d-1)^{2}-d_{1}(d-1-d_{1}) -\bigg(d^{2}-d - d_{1}(d-d_{1})-d_{3}\bigg) = d_{3}+d_{1}-d+1 = d_{3}-d_{2}+1.$$
For the second part of the proof, assume that $2d_{1}=d$. Then the defect of $\mathcal{L}$ is equal to 
$$\nu(C) = \bigg\lceil\frac{3}{4}(d-1)^{2}\bigg\rceil - \tau(C) = 3d_{1}^{2}-3d_{1} + 1 - (4d_{1}^2-2d_{1}-d_{1}^2 - d_{3}) = d_{3}-d_{1}+1.$$
Since $2d_{1} = d$ implies that $d_{1}=d_{2}$, we have
$$\nu(C) = d_{3}-d_{2}+1,$$
which completes the proof.
\end{proof}
\section{Combinatorial polynomials attached to line arrangements}
This section is the main contribution of our paper. We present an idea for attaching a combinatorial object to each $m$-syzygy line arrangement that behaves similarly to the Poincar\'e polynomial. Pokora recently has introduced this idea of constructing Poincar\'e-type polynomials, for instance, in the setting of conic-line arrangements with ordinary singularities \cite{Pok}, and we will use his idea in the setting of plus--one generated line arrangements. 

We will need the following two facts concerning a plus--one generated arrangement of $d$ lines with exponents $(d_{1}, d_{2}, d_{3})$, namely
\begin{enumerate}
    \item[a)] $d_{1}+d_{2}=d$, and
    \item[b)] $\sum_{r \geq 2}(r-1)t_{r} = d_{1}d_{2}+d_{3}.$
\end{enumerate}
For an $m$-syzygy arrangement $\mathcal{L}$ of $d$ lines and the exponents $(d_{1}, d_{2}, d_{3}, \ldots, d_{m})$, let us define the following polynomial:
\begin{equation}
\label{combp}
\mathcal{P}(\mathcal{L},d_{3};t) = 1 + dt + \bigg(\sum_{r\geq 2}(r-1)t_{r}-d_{3}\bigg)t^{2}.    
\end{equation}
It is worth noting that the polynomial depends on both the weak-combinatorics of $\mathcal{L}$ and the third exponent $d_{3}$, which, in general, \textbf{is not combinatorially determined}, as exponents of curves are not combinatorially determined, see \cite{CuntzPokora}.
Here is our first crucial observation.
\begin{proposition}
\label{cpog}
If $\mathcal{L}$ is a plus--one generated arrangement of $d$ lines with the exponents $(d_{1}, d_{2}, d_{3})$, then the polynomial $\mathcal{P}(\mathcal{L},d_{3};t)$ defined by \eqref{combp} splits over the rationals, and we have
$$\mathcal{P}(\mathcal{L},d_{3};t) = (1+d_{1}t)(1+d_{2}t).$$
\end{proposition}
\begin{proof}
By the assumption that $\mathcal{L}$ is plus--one generated, we can use the above mentioned properties $d_{1}+d_{2}=d$ and $\sum_{r\geq 2}(r-1)t_{r}-d_{3} = d_{1}d_{2}$, and hence
$$\mathcal{P}(\mathcal{L},d_{3};t) = 1 + (d_{1}+d_{2})t + d_{1}d_{2}t^{2} = (1+d_{1}t)(1+d_{2}t),$$
which completes the proof.
\end{proof}
We are going to use the above considerations to present a technique that can serve as a handy, combinatorial, non--plus--one generatedness criterion. To achieve this, we modify equation \eqref{combp} so that it becomes a purely combinatorial expression.
\noindent
Recall that if $\mathcal{L}$ is an $m$-syzygy arrangement of lines, then $d_m\leq d-2$, by \cite[Corollary 3.5]{Sch}. Moreover, using \cite[Theorem 2.1]{Dimca2}, we get
$$\frac{2d}{m(\mathcal{L})}-2 \leq d_{1} \leq d_{2} \leq d_{3} \leq \ldots \leq d_{m}\leq d-2,$$
and hence we have 
$$d_{3} \in \left\{ \bigg\lceil \frac{2d}{m(\mathcal{L})}-2 \bigg\rceil, \ldots , d-2\right\}.$$
The above constraints on $d_{3}$ are obviously purely combinatorial, and this observation allows us to construct the following combinatorial criterion.
\begin{theorem}[Non--plus--one generatedness criterion]
\label{npog}
Let $\mathcal{L}$ be an arrangement of $d$ lines that is not free. Define 
$$\mathcal{P}(\mathcal{L},h;t) = 1 + dt + \bigg(\sum_{r\geq 2}(r-1)t_{r}-h\bigg)t^{2}$$
where $h$ is a positive integer. If, for every $h \in \left\{ \lceil 2d / m(\mathcal{L}) -2 \rceil, \ldots , d-2\right\}$, polynomial $\mathcal{P}(\mathcal{L},h;t)$ does not split over the rationals with respect to the variable $t$, then $\mathcal{L}$ cannot be plus--one generated.
\end{theorem}
\begin{proof}
It follows directly from Proposition \ref{cpog}.
\end{proof}
Now we show how to use this result in practice. 
\begin{example}
Let $\mathcal{L}$ be an arrangement of $d=5$ lines and $t_{2}=10$. Now we need to find the range for $h$, and we can check that $h=3$ since $\frac{2d}{m(\mathcal{L})}-2 =\frac{10}{2}-2 = 3$ and $d-2=3$. Taking this into account, our combinatorial polynomial has the following form
$$\mathcal{P}(\mathcal{L},3;t) = 1 + 5t + 7t^{2},$$
and it does not split over the rationals and hence the arrangement cannot be plus--one generated. 
\end{example}
\begin{remark}
Obviously, our criterion is not perfect, as it requires that our polynomial not split over the entire range of admissible values of the third exponent $d_{3}$. However, this approach makes our tool purely combinatorial, independent of homological properties of line arrangements, which is the key asset of our approach.
\end{remark}
\section{Plus--one generated line arrangements constructed by a deletion technique}
In this section, we will demonstrate how to construct somehow non-obvious examples of plus--one generated line arrangements starting from the free ones. Let us recall that a general result by Abe \cite{Abe18} tells us that if $\mathcal{A} \subset \mathbb{P}^{2}$ is a free line arrangement, then the deletion arrangement $\mathcal{A}'$ obtained by removing one line $\ell \in \mathcal{A}$ is either free or plus--one generated. Using this idea, we will show that deletion arrangements obtained by removing one line from the Klein arrangement, the Wiman arrangements, and the dual Hesse arrangement are plus--one generated. For combinatorial descriptions of these arrangements, please consult \cite{bnc}. 
\begin{proposition}
\label{klein}
Let $\mathcal{K} \subset \mathbb{P}^{2}_{\mathbb{C}}$ be the Klein arrangement of $21$ lines and let $\ell \in \mathcal{K}$ be any line. Then the arrangement $\mathcal{K}' :=\mathcal{K} \setminus \{\ell\}$ is minimal plus--one generated with exponents $(9,11,12)$.
\end{proposition}
\begin{proof}
Let us recall that $\mathcal{K}$ is a free arrangement with exponents $(d_{1},d_{2}) = (9,11)$, as it is a reflection arrangement, see \cite[Chapter 6]{OT92}. Moreover, this arrangement admits $t_{4}(\mathcal{K})=21$ and $t_{3}(\mathcal{K})=28$, and on each line from the arrangement $\mathcal{K}$ there are exactly $4$ triple and $4$ quadruple points. Consider now $\mathcal{K}' = \mathcal{K} \setminus \{\ell\}$ with $\ell \in \mathcal{K}$ being any line. As we have already mentioned, $\mathcal{K}'$ is either free or plus--one generated. 
For this purpose, let us denote by $r = |\mathcal{K}' \cap \ell|$ and we denote by $\epsilon = \epsilon(\mathcal{K}',\ell)$ the invariant of the pair that depends on the fact whether or not the singularities are quasi-homogeneous. Since for line arrangements all singularities are quasi-homogeneous we have $\epsilon = 0$. According to \cite[Theorem 1.3 (3)]{POG}, $\mathcal{K}'$ is free if and only if $r \geq d_{1}' + 1 - \epsilon$, where $d_{1}' = {\rm mdr}(\mathcal{K}')$. Our problem boils down to checking whether $r \geq d_{1}' + 1 - \epsilon = d_{1}'+1$.
Let us estimate $d_{1}'$. Since $\mathcal{K}'$ admits only double, triple and quadruple points as intersections, then by \cite[Theorem 2.1]{Dimca1} we have
$$d_{1}' \geq \frac{1}{2}\cdot {\rm deg}(\mathcal{K}')-2 = 8,$$
which implies that
$$8 = r \geq d_{1}' + 1 \geq 8+1 = 9,$$
a contradiction. This means that $\mathcal{K}'$ is plus--one generated with exponents $(d_{1}',d_{2}',d_{3}')$, where
$$d_{1} =d_{1}'=9, d_{2} = d_{2}'=11, d_{3}' = {\rm deg}(\mathcal{K}')-r = 20 - 8 = 12.$$
Since $d_{3}'=d_{2}'+1$, this means that $\mathcal{K}'$ is minimal plus--one generated.
\end{proof}
\begin{remark}
Naturally, we wonder if we can construct new plus--one generated arrangements by applying different deletion strategies to the Klein arrangement. Consider the following deletion: remove one quadruple point and all lines passing through it. Since ${\rm PSL}(2,7)$ acts on the set of lines and all intersection points, we can test our strategy on any quadruple point. We denote the resulting deletion arrangement by $\mathcal{K}''$. We can verify that
$$t_{2}(\mathcal{K}'') = 16, \quad t_{3}(\mathcal{K}'') = 24, \quad t_{4}(\mathcal{K}'')=8.$$
Then
$$\mathcal{P}(\mathcal{K}'',h;t) = 1 + 17t + (88-h)t^{2}$$
with $h \in \{7, \ldots, 15\}$. We can easily check that for every admissible $h$ polynomial $\mathcal{P}(\mathcal{K}'',h;t)$ does not split over the rationals and hence $\mathcal{K}''$ cannot be plus--one generated.
\end{remark}
\begin{proposition}
Let $\mathcal{W} \subset \mathbb{P}^{2}_{\mathbb{C}}$ be the Wiman arrangement of $45$ lines, and let $\ell \in \mathcal{W}$ be any line. Then the arrangement $\mathcal{W}' := \mathcal{W} \setminus \{\ell\}$ is plus--one generated with exponents $(19,25,28)$.
\end{proposition}
\begin{proof}
Let us recall that $\mathcal{W}$ is a free arrangement with exponents $(d_{1},d_{2}) = (19,25)$, as it is a reflection arrangement, see \cite[Chapter 6]{OT92}. Moreover, this arrangement has $t_{3}(\mathcal{W})=120$, $t_{4}(\mathcal{W})=28$, $t_{5}(\mathcal{W})=36$, and on each line from the arrangement $\mathcal{W}$ there are exactly $8$ triple, $4$ quadruple points and $4$ quintuple points. Consider $\mathcal{W}' = \mathcal{W} \setminus \{\ell\}$, where $\ell \in \mathcal{W}$ is any line. We are going to use the same trick as above, namely by \cite[Theorem 1.3 (3)]{POG}, $\mathcal{W}'$ is free if and only if $r \geq d_{1}' + 1$, where $d_{1}' = {\rm mdr}(\mathcal{W}')$ and $r = |\mathcal{W}' \cap \ell|$.
Since $\mathcal{W}'$ admits only double, triple, quadruple and quintuple points as intersections, then by \cite[Theorem 2.1]{Dimca1} we have
$$d_{1}' \geq \bigg\lceil\frac{2}{5}\cdot {\rm deg}(\mathcal{W}')-2 \bigg\rceil = 16,$$
which implies that
$$16 = r \geq d_{1}' + 1 \geq 16 +1 = 17,$$
a contradiction. This means that $\mathcal{W}'$ is plus--one generated with exponents $(d_{1}',d_{2}',d_{3}')$, where
$$d_{1} =d_{1}'=19, d_{2} = d_{2}'=25, d_{3}' = {\rm deg}(\mathcal{W}')-r = 44 - 16 = 28.$$
\end{proof}

\begin{proposition}
Let $\mathcal{H} \subset \mathbb{P}^{2}_{\mathbb{C}}$ be the dual Hesse arrangement of $9$ lines and let $\ell \in \mathcal{H}$ be any line. Then the arrangement $\mathcal{H}' := \mathcal{H} \setminus \{\ell\}$ is plus--one generated with exponents $(4,4,4)$, hence $\mathcal{H}'$ is nearly free.
\end{proposition}
\begin{proof}
Recall that $\mathcal{H}$ is a free arrangement with exponents $(d_{1}, d_{2}) = (4, 4)$, as it is a reflection arrangement, see \cite[Chapter 6]{OT92}. We know that $t_3(\mathcal{H})=12$, and the arrangement is fully symmetric; that is, on each line of the arrangement there are exactly four triple points. Consider $\mathcal{H}' = \mathcal{H} \setminus \{\ell\}$, where $\ell \in \mathcal{H}$ is any line. Now we are going to use the same trick as above, namely by \cite[Theorem 1.3 (3)]{POG}, $\mathcal{H}'$ is free if and only if $r \geq d_{1}' + 1$, where $d_{1}' = {\rm mdr}(\mathcal{H}')$ and $r = |\mathcal{H}' \cap \ell|$.
Since $\mathcal{H}'$ admits only double and triple points as intersections, then by \cite[Theorem 2.1]{Dimca1} one has
$$d_{1}' \geq \bigg\lceil\frac{2}{3}\cdot {\rm deg}(\mathcal{H}')-2 \bigg\rceil = 4,$$
which implies that
$$4 = r \geq d_{1}' + 1 \geq 4 + 1 = 5,$$
a contradiction. This means that $\mathcal{H}'$ is plus--one generated with exponents $(d_{1}',d_{2}',d_{3}')$, where
$$d_{1} =d_{1}'= 4, d_{2} = d_{2}'= 4, d_{3}' = {\rm deg}(\mathcal{H}')-r = 8 - 4 = 4,$$
hence $\mathcal{H}'$ is nearly free.
\end{proof}
\begin{remark}
\label{pp1}
Now we will demonstrate how to apply our Poincar\'e-type polynomial, given by equation \eqref{combp}, in the context of the deletion technique. Let us focus on the dual Hesse arrangement $\mathcal{H}$ and its deletion $\mathcal{H}'$.
After removing any line, we can easily check that we have
$$t_{2}(\mathcal{H}')=4\quad \text{and}\quad t_{3}(\mathcal{H}')=8.$$
Then
$$\mathcal{P}(\mathcal{H}', h;t) = 1 + 8t + (20-h)t^2$$
with $h \in \{4,5,6\}$. We can check that $\mathcal{P}(\mathcal{H}', h;t)$ splits over the rationals provided that either $h=4$ or $h=5$. Observe that if $\mathcal{P}(\mathcal{H}', 4;t) = (1+4t)^2$, then we would get the exponents $(4,4,4)$, and if $\mathcal{P}(\mathcal{H}', 5;t) = (1+3t)(1+5t)$, then we would get the exponents $(3,5,5)$. However, the latter case is not possible because of \cite[Theorem 1.3 (3)]{POG} since $d_1=d_1'$ and $d_2=d_2'$. This example also explains how to find the exponents of the resulting plus--one generated arrangements combinatorially.
\end{remark}
\section{Minimal plus--one generated sporadic simplicial line arrangements}
In this final section, we present all minimal plus--one generated arrangements among sporadic simplicial line arrangements with up to $27$ lines.
Let us review the basic definitions related to simplicial line arrangements.

\begin{definition}
Let $\mathcal{L} \subset \mathbb{P}^{2}_{\mathbb{R}}$ be an arrangement of $d\geq 3$ lines such that $t_{d} = 0$, i.e., $\mathcal{L}$ is not a pencil of lines. We say that $\mathcal{L}$ is simplicial if the following equality holds:
\begin{equation}
t_{2} = 3 + \sum_{r\geq 4}(r-3)t_{r}
\end{equation}
\end{definition}
 It is worth noting here that this definition strictly follows from Melchior's inequality for line arrangements in the real projective plane \cite{Melchior}.

Many years ago Gr\"unbaum proposed a problem to classify all simplicial line arrangements and, somewhat surprisingly, this classification is still not complete! We know that there are $3$ infinite series and, additionally, there are examples that do not fit into these three series -- such non-fitting examples are commonly known as \text{sporadic} examples. There are about $90$ known such sporadic cases, at least according to the latest update on the subject presented in \cite{Cuntz}. Here, our aim is to detect, among all known sporadic simplicial arrangements of $d\leq 27$ lines, these arrangements that are minimal plus--one generated. Our decision to focus on that class of plus--one generated arrangements follows from the fact that we would like to understand the geometry of such examples and we have a very handy criterion that is presented in Theorem \ref{MPOG}. Our restriction to sporadic line arrangements up to $27$ follows from the fact that it is commonly accepted that their classification has been completed. Let us present our searching strategy to extract those sporadic simplicial line arrangements that are minimal plus--one generated.

First of all, for a given sporadic simplicial line arrangement $\mathcal{A}$ we determine its total Tjurina number, which is combinatorial, namely
\begin{equation}
\tau(\mathcal{L}) = \sum_{r\geq 2}(r-1)^{2}t_{r}(\mathcal{L}).
\end{equation}
Then we want to use directly the criterion given in Theorem \ref{MPOG}, namely for a given line arrangement $\mathcal{L}$ with $d$ lines we want to check whether the equation, with respect to the variable $r$, of the form
\begin{equation}
(d-1)^{2} - r(d-r-1) = \tau(\mathcal{L})+2
\end{equation}
has integer roots $r_{1},r_{2}$ such that one of these roots, say $r_{1}$, satisfies $r_{1} \leq d/2$. This step allows us to narrow down our search considerably. After some additional adjustments, we arrive at exactly $9$ cases, which we now want to examine, each case separately. We refer to Cuntz's classification of sporadic simplicial line arrangements \cite{Cuntz}, and we have adopted the way of listing these arrangements from the mentioned paper. Moreover, we represent the weak combinatorics of a given arrangement $\mathcal{A}(d,k)$ by the vector of the form $W_{d,k} = (d;t_{2}, ..., t_{m(\mathcal{A})})$.

\begin{itemize}
\item[$\mathcal{A}(14,3): $] The weak combinatorics is the following $W_{14,3} = (14;9, 16, 4, 1)$. The minimal free resolution of the Milnor algebra is the following form
\begin{align*}
0\rightarrow S(-22) \rightarrow S^{2}(-20)\oplus S(-21) \rightarrow S^{3}(-13) \rightarrow S,       
\end{align*}
 which means that $\mathcal{A}(14,3)$ is minimal plus--one generated with exponents $(7,7,8)$.

 \item[$\mathcal{A}(15,3): $] The weak combinatorics has the form $W_{15,3} = (15; 12,13,9)$. The minimal free resolution of the Milnor algebra is the following form
\begin{align*}  
0\rightarrow S(-24)\rightarrow S(-21)\oplus S(-22) \oplus S(-23) \rightarrow S^{3}(-14) \rightarrow S,
\end{align*}
 which means that $\mathcal{A}(15,3)$ is minimal plus--one generated with exponents $(7,8,9)$.

\item[$\mathcal{A}(15,5): $] The weak combinatorics has the form $W_{15,5} = (15; 9,22,0,3)$. The minimal free resolution of the Milnor algebra is the following form
\begin{align*}
       0\rightarrow S(-24)\rightarrow S(-21)\oplus S(-22)\oplus S(-23) \rightarrow S^{3}(-14) \rightarrow S,
\end{align*}
 which means that $\mathcal{A}(15,5)$ is minimal plus--one generated with exponents $(7,8,9)$.
 
 \item[$\mathcal{A}(16,7): $] The weak combinatorics has the form $W_{16,7} = (16;12,19,6,0,1)$. The minimal free resolution of the Milnor algebra is the following form
\begin{align*}
       0\rightarrow S(-25)\rightarrow S^{2}(-23)\oplus S(-24) \rightarrow S^{3}(-15) \rightarrow S,
\end{align*}
 which means that $\mathcal{A}(16,7)$ is minimal plus--one generated with exponents $(8,8,9)$.

  \item[$\mathcal{A}(18,6): $] The weak combinatorics has the form $W_{18,6} = (18;18,16,12,0,1)$. The minimal free resolution of the Milnor algebra is the following form
\begin{align*}
       0\rightarrow S(-28)\rightarrow S^{2}(-26)\oplus S(-27) \rightarrow S^{3}(-17) \rightarrow S,
\end{align*}
 which means that $\mathcal{A}(18,6)$ is minimal plus--one generated with exponents $(9,9,10)$.

   \item[$\mathcal{A}(18,8): $] The weak combinatorics has the form $W_{18,8} = (18;16,22,6,2,1)$. The minimal free resolution of the Milnor algebra is the following form
\begin{align*}
       0\rightarrow S(-29)\rightarrow S^{2}(-26)\oplus S(-27) \rightarrow S^{3}(-17) \rightarrow S,
\end{align*}
 which means that $\mathcal{A}(18,8)$ is minimal plus--one generated with exponents $(9,9,10)$.

    \item[$\mathcal{A}(19,7): $] The weak combinatorics has the form $W_{19,7} = (19;21,15,15,0,1)$. The minimal free resolution of the Milnor algebra is the following form
\begin{align*}
       0\rightarrow S(-30)\rightarrow S(-27) \oplus S(-28)\oplus S(-29) \rightarrow S^{3}(-18) \rightarrow S,
\end{align*}
 which means that $\mathcal{A}(19,7)$ is minimal plus--one generated with exponents $(9,10,11)$.

     \item[$\mathcal{A}(24,2): $] The weak combinatorics has the form $W_{24,2} = (24;32,32,0,12,0,0,1)$. The minimal free resolution of the Milnor algebra is the following form
\begin{align*}
       0\rightarrow S(-40)\rightarrow S(-32) \oplus S(-38)\oplus S(-39) \rightarrow S^{3}(-23) \rightarrow S,
\end{align*}
 which means that $\mathcal{A}(24,2)$ is minimal plus--one generated with exponents $(9,15,16)$.

      \item[$\mathcal{A}(24,3): $] The weak combinatorics has the form $W_{24,3} = (24;31,32,9,5,3)$. The minimal free resolution of the Milnor algebra is the following form
\begin{align*}
       0\rightarrow S(-38)\rightarrow S(-34) \oplus S(-36)\oplus S(-37) \rightarrow S^{3}(-23) \rightarrow S,
\end{align*}
 which means that $\mathcal{A}(24,3)$ is minimal plus--one generated with exponents $(11,13,14)$.
\end{itemize}
Summing up, we have the following result.
\begin{theorem}
\label{spog}
Among all known examples of sporadic simplicial arrangements up to $27$ lines from the catalogue presented in \cite{Cuntz}, there are exactly $9$ minimal plus--one generated arrangements, namely these arrangements are $\mathcal{A}(14,3)$, $\mathcal{A}(15,3)$, $\mathcal{A}(15,5)$, $\mathcal{A}(16,7)$, $\mathcal{A}(18,6)$, $\mathcal{A}(18,8)$, $\mathcal{A}(19,7)$, $\mathcal{A}(24,2)$, $\mathcal{A}(24,3)$.
\end{theorem}
\begin{remark}
\label{pp2}
Our investigations here can be supported theoretically using our Poincar\'e-type polynomial \eqref{combp}. Let us focus on arrangement $\mathcal{A}(13,3)$ from the catalogue \cite{Cuntz}, we have
$$W_{13,3} = (13; 10, 10, 3, 2).$$
We can compute the associated polynomial, namely
$$\mathcal{P}(\mathcal{A}(13,3),h;t) = 1+ 13t + (47-h)t^2,$$
with
$$4=\bigg\lceil\frac{2}{5}\cdot 13 -2 \bigg\rceil \leq h \leq d-2 = 11.$$
After checking all admissible values of $h$, we arrive at the following triplets of integers
$$(h;d_1, d_2) \in \{(5;6,7),(7;5,8),(11;4,9)\}.$$
Since the exponents are ordered, i.e., $d_{1} \leq d_{2} \leq d_{3}$ , we are left with $(h; d_{1}, d_{2}) = (11; 4, 9)$. To finish our investigation, we only need to determine the first exponent $d_{1}$. Using the defining equation of $\mathcal{A}(13,3)$ from \cite{Cuntz}, we can perform computations with \verb}SINGULAR} and calculate that $d_1=5$. Therefore, the arrangement $\mathcal{A}(13,3)$ cannot be plus--one generated.
\end{remark}
\section*{Acknowledgments}
I would like to thank Piotr Pokora for his guidance and an anonymous referee for useful suggestions that allowed me to improve the paper.

Artur Bromboszcz is supported by the National Science Centre (Poland) Sonata Bis Grant  \textbf{2023/50/E/ST1/00025}. For the purpose of Open Access, the author has applied a CC-BY public copyright license to any Author Accepted Manuscript (AAM) version arising from this submission.

\vskip 0.5 cm

\bigskip
Artur Bromboszcz,
Department of Mathematics,
University of the National Education Commission Krakow,
Podchor\c a\.zych 2,
PL-30-084 Krak\'ow, Poland. \\
\nopagebreak
\textit{E-mail address:} \texttt{artur.bromboszcz@uken.krakow.pl}
\end{document}